\documentclass{amsart}
\usepackage{graphicx}
\usepackage{verbatim}
\usepackage{wrapfig}

\usepackage{amsmath, amssymb, amsthm}
\usepackage{amsfonts}
\usepackage{comment}
\newtheorem{theorem}{Theorem}

\numberwithin{equation}{section}
\newtheorem{proposition}{Proposition}

\begin{document}
\title{Frustration in signed graphs}
\author{Vaidy Sivaraman}
\address{Department of Mathematical Sciences, 
Library North 2200, 
Binghamton University,
Binghamton, New York 13902-6000.}
\email{vaidy@math.binghamton.edu}

\date{March 27, 2014 \\  \text{      }  \text{  } \small {2010 Mathematics Subject Classification: 05C22}}   
\maketitle

\begin{abstract}
Zaslavsky conjectured the following: The minimum number of vertices to be deleted to restore balance in a subcubic signed graph is the same as the minimum number of edges to be deleted to restore balance. We prove this conjecture. Also, we obtain a bound for these invariants in a special class of signed graphs.  
\end{abstract}

In this note, we prove two results. The first result is a proof of a conjecture of Zaslavsky, and the second one is an upper bound  for the frustration index of a signed graph whose underlying graph is cubic and has girth at least $4$. \\
 
A signed graph $\Sigma$ is a pair $(G, \sigma)$, where $G$ is a graph and $\sigma : E(G) \to \{+1,-1\}$ is a function (called sign function). A circle in $\Sigma$ (connected nonempty $2$-regular subgraph) is said to be positive if the product of signs on its edges is $+1$, and negative otherwise. A signed graph is said to be balanced if all its circles are positive. Harary \cite{FH} introduced the notion of balance in a signed graph, and gave a characterization of balanced signed graphs. An unbalanced signed graph may become balanced when some of its vertices or edges are deleted. This is captured by two parameters: frustration number and frustration index.  Let $\Sigma$ be a signed graph. The {\it frustration index} of $\Sigma$, denoted $l(\Sigma)$, is the smallest number of edges whose deletion from $\Sigma$ leaves a balanced signed graph. The {\it frustration number} of $\Sigma$, denoted $l_0(\Sigma)$, is the smallest number of vertices whose deletion from $\Sigma$ leaves a balanced signed graph. We prove the conjecture of Zaslavsky \cite[Conjecture 7.1]{TZ} that these numbers are the same for any signed subcubic graph. 

Let $v$ be a vertex of $\Sigma$. The signed graph obtained from $\Sigma$ by negating all links incident to $v$ is said to be obtained from $\Sigma$ by switching at $v$. Two signed graphs are said to be switching equivalent if one can be obtained from the other by a sequence of switchings. 
% For a signed graph $\Sigma$, $l_0(\Sigma)$ denotes the minimum number of vertices to be deleted to restore balance, and  $l(\Sigma)$ denotes the minimum number of edges to be deleted to restore balance. 
 A signed graph is subcubic if every vertex in its underlying graph has degree at most $3$. Zaslavsky \cite{TZ} conjectured the following, which we prove.

\begin{theorem}\label{ZASLAVSKYCONJECTURE}
For every  signed subcubic graph $\Sigma$, $l_0(\Sigma) = l(\Sigma)$.  
\end{theorem}

\begin{proof}% [Proof of Theorem \ref{ZASLAVSKYCONJECTURE}]
Let $\Sigma$ be a signed subcubic graph. Since $\Sigma$ is subcubic, it can have at most one  negative loop incident to a vertex, and hence the presence of a negative loop increases both the frustration index and number by one (frustration number increases by one because if a negative loop is present, the vertex incident to it must be deleted, whereas if the negative loop is deleted, the deletion of the vertex (now isolated or pendant) does not affect the balance of the signed graph). Hence we may assume that $\Sigma$ is loopless. It is trivial to see that $l_0(\Sigma) \leq l(\Sigma)$. To see the other direction, suppose $l_0(\Sigma) = k$. Let $v_1, \ldots , v_k \in V(\Sigma)$ such that $\Sigma \backslash \{v_1, \ldots , v_k \}$ is balanced. Let $X = \{v_1, \ldots , v_k\}$ and $Y = V(\Sigma) - X$. Zaslavsky's lemma \cite{TZ1} says that two signed graphs with the same underlying graph have the same list of positive circles if and only if they are switching equivalent. Hence we can switch so that $\Sigma : Y$ consists of only positive edges.  If in the resulting signed graph, $v_i$ is incident to more than one negative edge for some $i$, $1 \leq i \leq k$, switch at $v_i$ to reduce the number 
of negative edges (since $\Sigma$ is subcubic). We keep doing this until we have a signed graph $\Sigma'$, switching equivalent to $\Sigma$, with at most one negative edge incident to each $v_i$, and since switching at a vertex in $X$ does not change the sign of an edge with both endpoints in $Y$, $\Sigma' :  Y$ consists of only positive edges. Let $S$ be the set of negative edges in $\Sigma'$. Since every edge in $S$ has an endpoint in $X$, and no two of them share an endpoint in $X$, $|S| \leq |X| = k$. Now, $\Sigma' \backslash S$ consists only of positive edges, and is therefore balanced. This implies $l(\Sigma) = l(\Sigma') \leq |S| \leq k = l_0(\Sigma)$. 
\end{proof}

Note that for general signed graphs $l$ can be arbitrarily large compared to $l_0$. For example, the graph obtained from $n$ triangles by identifying a vertex in each triangle, with one edge in each triangle negative has $l_0 = 1$ and $l = n$. But having a bound on the degree or knowing the degree sequence of the underlying graph can be useful. As Zaslavsky (private communication) points out, the above proof can be used to prove the following:

\begin{proposition}
Let $G$ be a graph with degree sequence $(d_1, \ldots ,d_n)$, where $d_1 \geq \ldots \geq d_n$. Let $\Sigma$ be a signed graph with underlying graph $G$. Then 
\begin{equation*}
l(\Sigma) \leq \displaystyle \sum_{i=1}^{l_0(\Sigma)} \left \lfloor \frac{d_i}{2} \right \rfloor.
\end{equation*}
\end{proposition}

What can we say about the frustration number of a signed subcubic loopless graph? Suppose that $\Sigma$ is a signed subcubic loopless graph. Then we can switch $\Sigma$ to a signed graph $\Sigma'$ such that the negative edges in $\Sigma'$ form a matching (cf. Corollary 7.4 \cite{TZ}). Hence we have $l_0(\Sigma) \leq \frac{|V(\Sigma)|}{2}$. Can this bound be improved? We give two families of signed cubic graphs to show that without any additional restrictions on the underlying graph, the bound cannot be improved. First family: Disjoint union of $k$ copies of $K_4$, with every edge negative. Second family: Disjoint union of $k$ copies of the signed graph consisting of two vertices and  three links, exactly one of them negative. \\

The following result tells us that the frustration index of a signed graph whose underlying graph has $n$ vertices and girth at least $4$, and is cubic is at most $\frac{3n}{8}$. 

\begin{theorem}\label{FRUSTRATIONINDEXBOUNDFORCUBICGRAPHS}
Let $\Sigma$ be a signed graph whose underlying graph is cubic and has girth at least $4$. Then $l(\Sigma) \leq \frac{3}{8} |V(\Sigma)|$. 
\end{theorem}

\begin{proof}
 We will ``reduce" $\Sigma$ by using the following two operations: 1) If the number of negative edges at a vertex $v$ is more than 1, switch at $v$. Since $\Sigma$ is cubic, this operation reduces the number of negative edges. 2) If there exist vertices $u,v,w$ such that $uv$ and $vw$ are positive edges, and each of the three vertices $u,v,w$ is incident to a negative edge, then switch at $\{u,v,w\}$. Note that the absence of triangles in $\Sigma$ ensures that the number of negative edges is decreased by operation 2. We keep doing these operations (operation 2 should be used only if operation 1 cannot be done) on $\Sigma$. Let the resulting signed graph be $\Sigma'$ (since both the operations reduce the number of negative edges, the process terminates after a finite number of steps). The fact that operation 1 cannot be done in $\Sigma'$ tells us that the negative edges form a matching. Let us denote the set of endpoints of the edges in the matching by $X$. The fact that operation 2 cannot be done in $\Sigma'$, together with the absence of loops, multiple edges, and triangles in the underlying graph of $\Sigma'$, tells us that the positive edges in $\Sigma':X$ form a matching. Since $\Sigma'$ is cubic,
 there is at least one edge from every vertex in $X$ to $V(\Sigma') - X$. Hence the number of edges between $X$ and $V(\Sigma')$ is at least $|X|$ and at most $3 |V(\Sigma') - X|$: 

\begin{equation*}
|X| \leq 3 |V(\Sigma') - X|.
\end{equation*}
This gives
\begin{equation*}
|X| \leq \frac{3}{4} |V(\Sigma')|.
\end{equation*}
Deleting exactly one of the endpoints of each of the negative edges in $\Sigma':X$ yields a balanced signed graph. Hence 
\begin{equation*}
l(\Sigma) = l(\Sigma') \leq \frac{3}{8} |V(\Sigma')|= \frac{3}{8} |V(\Sigma)|,
\end{equation*}
and this completes the proof. 
\end{proof}

The inequality in the previous theorem is tight. Consider the graph obtained from an octagon by adding diagonals between antipodal points (this graph is known in the literature as the Wagner graph). Declare the diagonal edges to be negative, and the edges on the octagon positive. It is easy to check that $l_0 = 3$ for this signed graph.

\begin{section}
{Acknowledgements}
The work presented here is part of my Ph.D. dissertation \cite{VS} written under the guidance of Neil Robertson at The Ohio State University. I would like to thank Daniel Slilaty and Thomas Zaslavsky for several helpful discussions. 
\end{section}


\begin{thebibliography}{999999}
\bibitem{AACE} J. Akiyama, D. Avis, V. Chv\'{a}tal, H. Era, Balancing signed graphs, Discrete Appl. Math., Volume 3, Issue 4, November 1981,  227-233.
%\bibitem{BBC} N. Bansal, A. Blum, S. Chawla, Correlation clustering, in: Proc. 43rd Ann. IEEE Sympos. Foundations of Computer Science, FOCS  '02, 238-247. 
%\bibitem {CH} D. Cartwright and F. Harary, Structural balance: a generalization of Heider's theory, Psychological Review, vol. 63 (1956), 277-293.
%\bibitem{JAD} J. A. Davis, Structural balance, mechanical solidarity, and interpersonal relations. Amer. J. Sociology {\bf 68} (1963), 444-463.
\bibitem {FH} F. Harary, On the notion of balance of a signed graph, Michigan Math. J. {\bf 2} (1953-54), 143-146. 
\bibitem{VS} V. Sivaraman, Some topics concerning graphs, signed graphs and matroids, Ph.D. Dissertation, The Ohio State University, 2012. 
\bibitem{TZ0} T. Zaslavsky, The geometry of root systems and signed graphs, Amer. Math. Monthly {\bf 88} (1981), no. 2, 88-105. 
\bibitem{TZ1} T. Zaslavsky, Signed graphs, Discrete Appl. Math. {\bf 4} (1982), no. 1, 47-74.  
\bibitem{TZ} T. Zaslavsky, Six signed Petersen graphs, and their automorphisms, Discrete Math. {\bf 312} (2012), no. 9, 1558-1583. 
\end{thebibliography}
\end{document}